\documentclass[11pt]{amsart}

\usepackage{amssymb,amsmath,amsthm,amsfonts, xypic}
\usepackage{hyperref} 

\newtheorem{theorem}{Theorem}[section]
\newtheorem{corollary}[theorem]{Corollary}
\newtheorem{lemma}[theorem]{Lemma}

\newtheorem{proposition}[theorem]{Proposition}

\newtheorem{remark}[theorem]{Remark}
\newcommand{\del}[2]{{}}

\newcommand{\C}{\mathbb C}

\newcommand{\R}{\mathbb R}
\newcommand{\Q}{\mathbb Q}
\newcommand{\Z}{\mathbb Z}

\newcommand{\gt}{\tilde \Gamma}

\def\O{{\mathcal O}}

\def\E{{\mathcal E}}
\def\a{{\mathfrak a}}

\def\p{{\mathfrak p}}
\def\g{{\mathfrak g}}
\def\P{{\mathbb P}}
\def\H{{\mathbb H}}
\def\sl{{\rm SL_2}}

\def\sl{{\rm SL_2}}
\def\psl{{\rm PSL_2}}
\def\gl{{\rm GL_2}}

\def\tr{{\rm tr}}

\title{On the Dimension of Cohomology of Bianchi Groups}

\author{Mehmet Haluk \c{S}eng\"{u}n}
\email{M.H.Sengun@warwick.ac.uk}
\urladdr{http://warwick.ac.uk/haluksengun}
\address{Mathematics Institute, University of Warwick, Coventry, UK}

\author{Seyfi T\"urkelli}
\email{s-turkelli@wiu.edu}
\urladdr{http://www.wiu.edu/users/st110/}
\address{\mbox{\footnotesize{Department of Mathematics,} Western Illinois University, \footnotesize{Macomb}, USA}}  

\begin{document}

\maketitle

\begin{abstract}
Using Lefschetz numbers of certain involutions, we provide explicit lower bounds for the cuspidal cohomology 
of principal congruence subgroups of Bianchi groups. The asymptotic lower bounds that follow from 
our results complement recent results of Calegari-Emerton, Marshall and Finis-Grunewald-Tirao.   
\end{abstract}
 
\section{Introduction}
Bianchi groups are groups of the form $\sl(\O)$ where $\O$ is the ring of integers of an imaginary 
quadratic field. Just as the cohomology of the classical modular group $\sl(\Z)$ is central to 
the theory of classical modular forms, the cohomology of Bianchi groups is central to the study of 
Bianchi modular forms, that is, modular forms over imaginary quadratic fields. 

Understanding the behavior of the dimension of the cohomology of Bianchi groups and their congruence subgroups is a long open problem. Up to date, there is no explicit dimension formula of any sort. Utilizing the compactification theory of Borel-Serre (which basically amounts to closing the cusps of the 3-folds associated to Bianchi groups with 2-tori), we can decompose the cohomology into two parts: the cuspidal part and the Eisenstein part. While it is easy to compute the dimension of the Eisenstein part, understanding the dimension of the cuspidal part is very hard.

In 1984 Rohlfs, developing an idea that goes back to Harder (see the end of \cite{harder-75}), provided in \cite{rohlfs-84} an explicit lower bound for (the cuspidal part of) the first cohomology with trivial complex coefficients of Bianchi groups. Around the same time, Kr\"amer, mainly using techniques developed by Rohlfs, made these lower bounds sharper. In their recent paper \cite{fgt}, Finis, Grunewald and Tirao provided explicit lower bounds for the cuspidal part of the first cohomology with certain non-trivial coefficient systems of Bianchi groups. 

There has been significant recent developments in understanding the behavior of the dimension asymptotically. In \cite{calegari-emerton} Calegari and Emerton, using techniques from non-commutative Iwasawa theory, provided asymptotic upper-bounds for the first cohomology, with a fixed coefficient system, as one goes down in a tower of principal congruence subgroups of prime-power level of a fixed Bianchi group.  In a complementary direction, Marshall proved in \cite{marshall}, using the approach of Calegari and Emerton,  an asymptotic upper-bound for the first cohomology of a congruence subgroup of a Bianchi group as the coefficient system varies. 

In this paper, we provide explicit lower bounds for the cuspidal cohomology of principal congruence subgroups of Bianchi groups, using results and methods produced by Rohlfs \cite{rohlfs-78} and Blume-Nienhaus \cite{blume}. As a by product, we derive asymptotic lower bounds, complementing the above mentioned work of Calegari-Emerton and Marshall. A summary of the method and the results is provided in Section \ref{summary} for the convenience of the reader.  We discuss the case of general involutions in Section \ref{section_LFPT}, we specialize to the involutions induced by complex conjugation and twisted complex conjugation in Section \ref{section_LNSI}. In Section \ref{section_eisenstein}, we discuss the contribution of the cohomology of the boundary. The asymptotic lower-bounds are presented in Section 
\ref{section_asymptotics}. 

Our main contribution is the computation of the traces of the above mentioned special involutions acting on the Eisenstein part of 
the cohomology of the principal congruence subgroups. The case of the first cohomology is especially hard. In this case, we employ the explicit cocycles 
of Sczech and utilize results of Ito to obtain our results. While the concrete nature of the Sczech cocyles allows an elementary and quick treatment, 
it also limits the scope of our results. In principle, Harder's theory of the Eisenstein cohomology could be used to obtain more complete results 
at the cost of losing the elementary nature of Sczech cocyles' setting.

{\bf Acknowledgments} We thank Steffen Kionke and Joachim Schwermer, who recently announced some asymptotic lower 
bounds for the first Betti numbers of arithmetic hyperbolic 3-manifolds,  for bringing to our attention a mistake in an earlier 
version of this paper. The first author thanks the Max Planck Institute for Mathematics for the hospitality he received during the 
the preparation of this paper.

\subsection{Set-up} 
Fix a square-free negative integer $d \not = -1,-3$, let $K$ be the imaginary quadratic field $\Q(\sqrt{d})$ with class number $h$ and ring of integers $\O$. Let $G$ be the associated Bianchi group $\sl (\O)$ and $\Gamma$ be a finite index subgroup of $G$. Given a nonnegative integer $k$, let $E_k$ be the space of homogeneous polynomials over $\C$ in two variables of degree $k$  with the following $\Gamma$-action: given a polynomial $p(x,y)\in E_k$,  
$$p(x,y)\cdot \left(\begin{smallmatrix} a&b \\ c&d\end{smallmatrix}\right)=p(ax+by,cx+dy).$$ 
Let $E_{k,k}:=E_k\otimes_{\C} \overline{E_k}$ is a $\Gamma$-module where the action of $\Gamma$ on the second component is twisted by the conjugation. 
 
 The group $G$ acts discontinuously as isometries on the hyperbolic 3-space $\H \simeq \C \times \R^+$ and the quotient $Y_\Gamma := \Gamma\backslash\H$ has the structure of an hyperbolic $3$-fold. Let $\E$ be the local system on $Y_\Gamma$ induced by some complex finite-dimensional $\Gamma$-representation $E$. It is well known that $Y_\Gamma$ is an Eilenberg-MacLane space for $\Gamma$ and so $$H^n(\Gamma,E)\cong H^n(Y_\Gamma,\E).$$

%----------------------------------------------------------------

\subsection{Summary of Results}\label{summary}

 Let $\sigma\in G(K/\Q)$ be the only  nontrivial element; that is, the complex conjugation. Suppose that  $\sigma$ acts on $E$  and  $\Gamma$ in a compatible way so that it induces an action on the cohomology $H^i(\Gamma, E)$. Since $\sigma$ is an involution, the eigenvalues of this action is $\pm 1$ and so the trace $\tr(\sigma\mid H^i(\Gamma,E))$ is an integer. 
 
 One defines the \emph{Lefschetz number} of $\sigma$ as the following integer 
$$L(\sigma, \Gamma, E)=\sum_{i} (-1)^i \tr (\sigma \mid H^i(\Gamma, E)).$$
 
These Lefschetz numbers were first considered by Harder in \cite{harder-75} where he computed them to give lower bounds for the cohomology of certain types of principal congruence subgroups $\Gamma$ with $E=\C$. In his 1976 Bonn Habilitation Rohlfs developed tools to compute these Lefschetz numbers for general arithmetic groups. In 1984, Rohlfs used these tools to provide lower bounds for the Lefschetz number for the case $\Gamma=\sl(\O)$ and $E=\C$. Later that year, in his Bonn Ph.D. thesis, Kr\"amer gave a closed formula for the Lefschetz number for the same case. These techniques were further developed by Blume-Nienhaus 
in his 1992 Bonn Ph.D. thesis where he provided the Lefschetz numbers for general $E_{k,k}$. The following, see Proposition \ref{prop_lefprin}, is an analogue of his results for principal congruence subgroups.
 
  \begin{proposition}\label{thm_lefprincipal}
 Let $N>2$ be a positive integer and $\Gamma(N)$ be the principal congruence subgroup of $\sl(\O)$ of level $(N)$. Then
$$L(\sigma,\Gamma(N),E_{k,k})=(A + 2B)   \dfrac{-N^3}{12} \prod_{p|N} (1-p^{-2}) \cdot (k+1)$$ 
where $A,B$ are explicit constants depending on the ramification data of $K/\Q$.
 \end{proposition}

 The constants $A$ and $B$ are in fact certain powers of $2$ and they were computed by Rohlfs in \cite{rohlfs-78}. These constants vary depending on the ramification data of our imaginary quadratic field $K$ and the ideal $\a$.

Following Harder, we use the trace of the involution $\sigma$ on $H^i(\Gamma, E)$ to bound the dimension of this cohomology space from below. In order to carry this idea out, one needs to calculate the trace of $\sigma$ on the Eisenstein part of the cohomology as well. The following theorem generalizes a part of the results announced by Harder at the very end of \cite{harder-75}, see Theorem \ref{trace2_eisenstein}.

\begin{theorem} \label{thm_trace2eis}
 Let $t$ be the number of distinct prime divisors of the discriminant of $K/\Q$. Let $N=p_1^{n_1} \hdots p_r^{n_r}$ be a positive integer whose prime divisors $p_i$ are unramified in $K$ and let $\Gamma=\Gamma(N)$ be the associated principal congruence subgroup of the Bianchi group $\sl(\O)$.
  
We have $$\tr(\sigma\mid H^2_{Eis}(\Gamma, E_{k,k}))=-2^{t-1} \cdot \prod_{i=1}^r (p_i^{2n_i}-p_i^{2(n_i-1)}) +\delta(0,k),$$ where $\delta$ is the Kronecker $\delta$-function, in other words, $\delta(0,k)=0$ unless $k=0$ in which case $\delta(0,k)=1$. In particular, $$\tr(\sigma\mid H^2_{Eis}(\sl(\O), E_{k,k}))=-2^{t-1}+\delta(0,k).$$
\end{theorem}

Computing the trace on the Eisenstein part of the first cohomology is more challenging as reported by Harder in \cite{harder-75}. He does not provide a proof but informs us that he uses the adelic setting and representation theoretic approach for his computations and his final result depends on certain factors in the functional equation of associated Hecke $L$-series. We provide a partial generalization of Harder's result, using an elementary approach which employs the cocycles of Sczech, see \cite{sczech}. These cocycles are defined by means of certain elliptic analogues of classical Dedekind sums, see Theorem \ref{trace1_eisenstein}

\begin{theorem} \label{thm_trace1eis}
Assume that $K$ is of class number one and let $p$ be a rational prime that is inert in $K$. Then we have 
$$ \tr(\sigma\mid H^1_{Eis}(\Gamma(p^n),\C)) = \begin{cases} -(p^2+1), \ \ \ \ \ \ \ \ \ \ \textrm{if} \ n=1 \\ 
                                                                              -(p^{2n}-p^{2n-2}), \ \ \ \textrm{if} \ n>1. \end{cases}$$ 
\end{theorem}

We believe that the above result should generalize higher class numbers as well.  
Our results so far allow us to get explicit lower bounds for the cuspidal cohomology of certain principal congruence subgroups that are stabilized by complex conjugation. These explicit lower bounds yield the following asymptotic bounds, Proposition \ref{prop_asym} below.  For a related result, see the article \cite{rohlfs-speh} of Rohlfs and Speh. 

\begin{corollary}\label{cor_asym}
 Let $p$ be a rational prime that is unramified in $K$ and let $\Gamma(p^n)$ denote the  principal congruence subgroup of level $(p)^n$ of a Bianchi group $\sl(\O)$. Then, as $k$ increases and $n$ is fixed 
$$ \dim H^1_{cusp}(\Gamma(p^n),E_{k,k})\gg k$$ 
where the implicit constant depends on the level $\Gamma(p^n)$ and the field $K$. 
Assume further that $K$ is of class number one and that $p$ is inert in $K$. Then, as $n$ increases 
$$ \dim H^1_{cusp} (\Gamma(p^n),\C)\gg p^{3n}$$ 
where the implicit constant depends on the field $K$. 
\end{corollary}

We also consider the Lefschetz numbers and the Eisenstein traces for the involution given by  the $\gl / \sl$-twist of complex conjugation. The results, when combined with those about complex conjugation, give a closed formula for the trace of $\sigma$ on the first cohomology of $\gl(\O)$, see Theorem \ref{thm_tracegl}. This implies the following asymptotics for the cohomology of $\gl(\O)$, see Corollary \ref{cor_asymgl}.

\begin{corollary}\label{cor_asymglintro}
 Let $D$ be the discriminant of $K/\Q$ and $\O_K$ be its ring of integers. As $K/\Q$ is fixed and $k\to \infty$, we have $$\dim H^1(\gl (\O_K),E_{k,k}) \gg k$$ where the implicit constant depends on the discriminant $D$. As $k$ is fixed and $|D|\to \infty$, we have $$\dim H^1(\gl(\O_K),E_{k,k}) \gg \varphi(D)$$ where $\varphi$ is the Euler $\varphi$ function and  the implicit constant depends on the weight $k$. 
\end{corollary}

As $H^1(\gl(\O),E_{k,k})$ embeds into $H^1_{cusp}(\sl(\O),E_{k,k})$, the asymptotic lower bounds of the above corollary also applies to $H^1_{cusp}(\sl(\O),E_{k,k})$. Rohlfs showed in \cite{rohlfs-84} that $H^1_{cusp}(\sl(\O),\C) \gg \varphi (D)$ as $|D| \rightarrow \infty$, yielding the same asymptotic as ours. We note that Kr\"amer also 
produces the upper bound $$ \dim H^1_{cusp}(\sl(\O),\C) \ll |D|^{3/2}.$$

 %---------------------------------------------------------------------------------------------------------------
%---------------------------------------------------------------------------------------------------------------

\section{A Lefschetz fixed point theorem} \label{section_LFPT}
Let $\g=\langle \rho\rangle$ be a finite cyclic subgroup  of order $r$ of the automorphism group $\textrm{Aut}(G)$ of the Bianchi group 
$G$ (note that $\textrm{Out}(G)$ is finite elementary abelian 2-group which is explicitly determined by Smillie and Vogtmann in \cite{smilie-vogtmann}). Let $\Gamma$ be 
a $\g$-stable finite index subgroup of $G$ considered as a normal subgroup of the semidirect product  $\gt=\Gamma \rtimes \g $. The group $\gt$ has a natural action on $\H$ that extends the action of $\Gamma$ (thus, $\gt$ acts on $Y_{\Gamma}$).

 Let $E$ be a $\Gamma$-module with a $\g$-action such that this action is compatible with the action on $\Gamma$, that is, 
${}^{\rho}(g \cdot e) = {}^\rho g \cdot {}^\rho e$. Then $\g$ acts on the cohomology groups $H^i(\Gamma,E)$. Therefore, we can define the \emph{Lefschetz number} 
$$L(\rho, \Gamma, E)=\sum_i (-1)^i \tr ( \rho\mid H^i(\Gamma,E)).$$ 
  
  Given a subgroup $H\subseteq \gt$, let $\chi (H)$ denote the \emph{virtual Euler-Poincare} characteristic of $H$. If $H$ is a finite group, then it is well known that $\chi (H)=1/|H|$. Below, for $\gamma\in \Gamma$, by ``$\gamma\rho \ \ \mod \Gamma$" we mean the $\Gamma$-conjugacy class of $\gamma\rho$ in $\gt$. One of the main results of the Bonn PhD thesis of  Blume-Nienhaus \cite{blume}, which generalizes the Lefschetz trace formula to not necessarily torsion-free arithmetic groups, tells us the following.
  
\begin{theorem}[Blume-Nienhaus, I. 1.6]\label{thm_ltfblume}
Using the notation above, we have $$L(\rho,\Gamma,E)=\sum_{\substack{\gamma\rho \space \mod \Gamma \\ \gamma\rho \text{ is torsion} \\ \gamma\in \Gamma}} \chi(\Gamma^{\gamma\rho}) \cdot\tr(\gamma\rho\mid E)$$
where $\Gamma^{\gamma\rho}$ denotes the centralizer in $\Gamma$ of $\gamma\rho$.
\end{theorem}

The relationship between Lefschetz fixed point formula and centralizers of torsion elements was observed also by Adem in \cite{adem}.
For our special representations $E_{k,k}$, the traces involved in the above formula can  easily be computed. 
 
 \begin{lemma}[Blume-Nienhaus, I.4.3]\label{lem_trreform}
 Let $\gamma\in \Gamma$ and $x=(\gamma\rho)^r$, where $r$ is the order of $\rho$. Then, $$\tr(\gamma\rho\mid E_{k,k})=\tr(x\mid E_k).$$
 \end{lemma}
 
%----------------------------------------------------------------------------------------
 \subsection{Shapiro's Lemma} \label{SL}
Using Shapiro's Lemma, we can relate the Lefschetz number $L(\rho,\Gamma,E)$ to the Lefschetz number 
$L(\rho,G,\textrm{Coind}_{\Gamma}^G(E))$ where 
$$\textrm{Coind}_\Gamma ^G (V) := \textrm{Hom}_{\Gamma}(G,V).$$
We define the $\rho$-action on the co-induced module
as follows:
$$ ({}^{\rho}f) (g) := {}^{\rho}(f({}^{\rho}g))$$
for every $f \in \textrm{Coind}_{\Gamma}^{G}(V)$ and $g \in G$. This action is compatible with the action of $\rho$ on $G$. 
Moreover, the Shapiro isomorphism respects the $\rho$-action. 

\begin{proposition}\label{shapiro} The isomorphism 
$$H^i(\Gamma,E) \simeq H^i(G,\textrm{Coind}_{\Gamma}^G(E))$$
respects the involutions induced by $\rho$ on both sides.
\end{proposition}

\begin{proof} Let $s: E \rightarrow \textrm{Coind}_\Gamma ^G (E)$ be defined as $s(e)(g)=g \cdot e $ if $g \in \Gamma$ and 
$s(e)(g)=0$ otherwise. Then we have an induced map 
$$s^*: H^i(\Gamma, E) \rightarrow H^i(\Gamma, \textrm{Coind}_\Gamma ^G(E)).$$
It is easy to check that $s^*$ respects the involutions induced by $\rho$ on both sides. Now, the Shapiro isomorphism is given 
as 
$$H^i(\Gamma, E) \xrightarrow{s^*} H^i(\Gamma, \textrm{Coind}_\Gamma ^G(E)) \xrightarrow{ \textrm{cores}} H^i(G, \textrm{Coind}_\Gamma ^G(E)).$$
Recall that the corestriction map ``cores' can be given explicitly at the level of cocycles as follows. Let $\{ \gamma_i \}_i$ be a set of coset representatives of 
$\Gamma$ in $G$. Then for every $g \in G$ and there is a unique permutation $\phi_g$, that depends on $g$ and the $\{ \gamma_i \}_i$, 
such that  
$\gamma_{\phi_g(i)}^{-1}g \gamma_i \in \Gamma$. For a cocycle $f : \Gamma \rightarrow \textrm{Coind}_\Gamma ^G(E)$, 
we have 
$$ cores(f)(g) := \sum_i f(\gamma_{\phi_g(i)}^{-1}g \gamma_i)$$
for every $g \in G$.

To see that corestriction map respects the involutions induced by $\rho$ on both sides, note that for a given cocycle $f$, 
the class of $cores(f)$ is independent of the coset representatives $\{ \gamma_i \}$ that we chose. If $\{ \gamma_i \}_i$ 
is a set of coset representatives of $\Gamma$ in $G$, then $\{ \delta_i :={}^\rho \gamma_i \}_i$ is also such a set. For a given 
$g \in G$, let $\psi_g$ be the associated permutation with respect to the $\{ \delta_i \}_i$. Since $\Gamma$ is 
$\rho$-stable, we have $\phi_g(i) = \psi_{{}^\rho g}(i)$. It follows that 

\begin{eqnarray*}
cores({}^\rho f)(g) & = &  \sum_i ({}^\rho f)(\gamma_{\phi_g(i)}^{-1}g \gamma_i )  \\
                             &= &  \sum_i {}^\rho \left ( f({}^\rho \gamma_{\phi_g(i)}^{-1} {}^\rho g {}^\rho \gamma_i) \right ) \\  
                             &=& {}^\rho \left (     \sum_i f( \delta^{-1}_{\psi_{{}^\rho g}(i)}  {}^\rho g \delta_i  ) \right )  \\
                             &=^*&( {}^\rho cores(f))(g) \\ 
\end{eqnarray*}
The symbol $=^*$ means that the equality is up to a coboundary, this is because we have made a change of coset representatives 
in the expression of the corestriction map. This shows that at the level of cohomology classes, the corestriction map respects the involution 
that is induced by $\rho$. 
\end{proof}

\begin{remark} It is well-known that Shapiro's Lemma respects the action of Hecke operators. 
Therefore the above proposition holds when $H^i$ is replaced by $H^i_{cusp}$ and $H^i_{Eis}$ as well (see Section \ref{subsec_boundsviaLefschetz}). 
\end{remark}

For trace computations it is more convenient to replace $\textrm{Coind}_\Gamma ^G(E)$ with
$\C[\Gamma \backslash G] \otimes_\C E$. On the latter, the $G$-action is diagonal and the $\rho$-action is given as ${}^\rho [g\Gamma, e] := [{}^\rho g \Gamma, {}^\rho e]$. Fix a set of coset representatives $\{ \gamma_i \}_i$ of $\Gamma$ in $G$. The explicit identification is given as follows. 
$$\textrm{Coind}_\Gamma ^G(E) \xrightarrow{F_1} \C[G] \otimes_{\Gamma} E \xrightarrow{F_2} \C[\Gamma \backslash G] \otimes_\C E$$
where 
$$F_1(f) := \sum_i \gamma_i \otimes f(\gamma_i) \ \ \ \textrm{and} \ \ \ F_2( g \otimes e) := [ g\Gamma, g e] .$$
One can check that the above maps give rise to isomorphisms
$$H^i(G, \textrm{Coind}_\Gamma ^G(E)) \simeq H^i(G,\C[\Gamma \backslash G] \otimes_\C E) $$
which respect the action of the involution induced by $\rho$ on both sides.

\begin{lemma} 
 Let $\gamma\in G$ and $x=(\gamma\rho)^r$, where $r$ is the order of $\rho$. Then 
$$\tr(\gamma \rho\mid \C[\Gamma \backslash G] \otimes_\C E_{k,k})=\tr(\gamma \rho \mid \C[\Gamma \backslash G]) \cdot \tr(x\mid E_k)$$
\end{lemma}
\begin{proof} It is clear that 
$$\tr(\gamma \rho\mid \C[\Gamma \backslash G] \otimes_\C E_{k,k})=\tr(\gamma \rho\mid \C[\Gamma \backslash G]) \cdot 
\tr(\gamma \rho\mid  E_{k,k}).$$
By Lemma \ref{lem_trreform}, we have $\tr(\gamma \rho\mid E_{k,k}) = \tr(x \mid E_k)$. 
\end{proof}

 %---------------------------------------------------------------------------------------

%---------------------------------------------------------------------------

\subsection{Torsion-free groups} \label{sec_ltffree}
When $\Gamma$ is torsion-free, one can give a geometric description of the Lefschetz trace formula. 

 Let $Y_\Gamma=\H/\Gamma$ and let $Y_{\Gamma}^{\rho}$ be the set of fixed points of the $\rho$-action on $Y_{\Gamma}$. Let  $\E^{\rho}$ denote the restriction of the sheaf $\E$ to $X_{\Gamma}^{\rho}$. Then  $\rho$ acts on the stalk of $\E^{\rho}$ and $L(\rho,Y_{\Gamma}^{\rho},\E^{\rho})$ is defined. We have the following geometric reformulation of the Lefschetz number.

\begin{proposition}  Assume that $\Gamma$ is torsion-free. Then 
$$L(\rho,\Gamma, E) = L(\rho,Y_{\Gamma}^{\rho},\E^{\rho}).$$
\end{proposition}

\begin{proof}
A proof is provided by Rohlfs and Schwermer in \cite{rohlfs-schwermer} page 152.
\end{proof}

 When $\Gamma$ is torsion-free, the connected components of $Y_{\Gamma}^{\rho}$ can be parametrized by the first non-abelian (Galois) cohomology $H^1(\mathfrak{g},\Gamma)$.   If $\gamma$ is a cocycle for $H^1(\mathfrak{g},\Gamma)$, we have a $\gamma$-twisted $\rho$-action on $\H$ given by $x \mapsto {}^{\rho}x \gamma^{-1}$. The fixed point set $\H(\gamma)$ of the $\gamma$-twisted action on $\H$ is non-empty and  its image in $Y_{\Gamma}$ is contained in $Y_{\Gamma}^{\rho}$. Let us denote this image by $F(\gamma)$. 

\begin{proposition}\label{prop_dec}
 Assume that $\Gamma$ is torsion free. The set of fixed points $Y_{\Gamma}^{\rho}$ is a finite disjoint union of its connected components $F(\gamma)$: $$Y_{\Gamma}^{\rho} = \bigcup_{\gamma \in H^1(\mathfrak{g},\Gamma)} F(\gamma).$$ The $F(\gamma)$'s are locally symmetric subspaces of $Y_{\Gamma}$.
\end{proposition}

In the presence of torsion in $\Gamma$, the above is not true: the left hand side is larger.

\begin{proposition}{\cite[Lemma 2.4.2]{rohlfs-84}}
 For $\Gamma$ not necessarily torsion-free, the difference $$Y_\Gamma^\rho \quad\backslash \bigcup_{\gamma\in H^1(\mathfrak{g},\Gamma)} F(\gamma)$$ is a finite set  which contains singular points of $Y_\Gamma$.
\end{proposition}

 There is also a $\gamma$-twisted $\rho$-action on $\Gamma$ given by  $g \mapsto \gamma~{}^{\rho}g \gamma^{-1} $ for $g \in \Gamma$. Let $\Gamma(\gamma)$ denote the set of fixed points of this action. When $\Gamma$ is torsion-free, the canonical map $$ \pi_{\gamma} : \Gamma(\gamma) \backslash \H(\gamma) \rightarrow Y_{\Gamma}$$ is injective. The image of $\pi_{\gamma}$ is homeomorphic to $F(\gamma).$ 
 
 There is a twisted $\rho$-action on $E$ as well, given by $e \mapsto  {}^\rho e \gamma$ for $e \in E$. The trace of this action on $E$ does not depend on the choice of the cocycle $\gamma$ in its class and therefore will be written as $\tr(\rho_{\gamma} \mid E)$. When $\Gamma$ is torsion-free,  the contractibility of $\H(\gamma)$ implies that $\chi(F(\gamma)) = \chi(\Gamma(\gamma))$. It follows that, see \cite{gmasz} p.26 for a proof, $$ L(\rho, F(\gamma), \E) = \chi(\Gamma(\gamma)) \cdot \tr(\rho_\gamma\mid E).$$ Hence, we get the following geometric reformulation of the Lefschetz trace formula for the torsion-free case.

\begin{theorem}[Rohlfs]\label{thm_lef}  Assume that $\Gamma$ is torsion-free. Then 
$$L(\rho,\Gamma,E) = \sum_{\gamma \in H^1(\mathfrak{g}, \Gamma)} \chi(F(\gamma)) \tr(\rho_\gamma\mid E).$$
\end{theorem}

%-------------------------------------------------------------------------

\subsection{Lower bounds for the cohomology via Lefschetz numbers}\label{subsec_boundsviaLefschetz}
 
 For the rest of the section, assume that $\rho$ is {\em orientation-reversing}, as it will be the case with the specific involutions that we will work with in Section 3.  
In this section, we want to give a lower bound for the dimension of the cuspidal cohomology in terms of the Lefschetz number of $\rho$. 

 Let $X_\Gamma$ denote the Borel-Serre compactification of $Y_\Gamma$. This is a compact manifold with boundary whose interior is homeomorphic to $Y_{\Gamma}$. Moreover, the embedding $Y_{\Gamma} \hookrightarrow X_{\Gamma}$ is homotopy equivariant, giving an isomorphism $$H^i(Y_{\Gamma}, \E) \simeq H^i(X_{\Gamma}, \bar{\E})$$ where $\bar{\E}$ is a certain sheaf on $X_{\Gamma}$ that extends $\E$.

Consider the long exact sequence$$ \hdots \rightarrow H^{i-1}_{c}(X_{\Gamma},\bar{\mathcal{E}}_{n}) \rightarrow H^{i}(X_{\Gamma},\bar{\mathcal{E}}_{n}) \rightarrow H^{i}(\partial X_{\Gamma},\bar{\mathcal{E}}_{n}) \rightarrow \hdots $$ here $H^i_c$ denotes the compactly supported cohomology.

The \textit{cuspidal cohomology} $H^i_{cusp}$ is defined as the image of the compactly supported cohomology. The \textit{Eisenstein cohomology} $H^i_{Eis}$ is the complement of the cuspidal cohomology inside $H^i$ and it is isomorphic to the image of the restriction map inside the cohomology of the boundary. Assume that the action of $\rho$ on $Y_{\Gamma}$ extends to $X_{\Gamma}$, which will be the case for our 
specific involutions of Section 3.  This induces involutions on the terms of the above long exact sequence. We therefore have, in the 
obvious notation, that $$\tr(\rho^i) = \tr(\rho^i_{cusp}) + \tr(\rho^i_{Eis}).$$

Poincar\'e duality implies that $H^1_{cusp} \simeq H^2_{cusp}$. Since $\rho$ is an orientation reversing involution, it follows that $\tr(\rho^1_{cusp})=-\tr(\rho^2_{cusp})$. Hence we get 
$$L(\rho,\Gamma,E) = \tr(\rho^0)-2 \tr(\rho^1_{cusp}) - \tr(\rho^1_{Eis}) + \tr(\rho^2_{Eis}), $$ and this implies the following proposition.

\begin{proposition}\label{lefschetzlowerbound}
With the above notation, we have
$${\rm dim} H^1_{cusp}(\Gamma, E) \geq \dfrac{1}{2}\bigg | L(\rho,\Gamma,E) +\tr(\rho^1_{Eis}) - \tr(\rho^2_{Eis})-\tr(\rho^0) \bigg |.$$
\end{proposition}

\begin{proof}
Since $\rho$ is an involution, the eigenvalues of $\rho^1_{cusp}$ are $\pm 1$, and so $${\rm dim}H^1(\Gamma,E)\geq  |\tr(\rho^1_{cusp})|.$$ The result now follows from the identity above.
\end{proof}

Note that when $E=E_{k,k}$ with $k>0$, $\tr(\rho^0)=0$ as $E$ is an irreducible $\Gamma$-representation.

%-------------------------------------------------------------------------------------------------------------------------------------
%------------------------------------------------------------------------------------------------------------------------------------
%---------------------------------------------------------------------------------------------------------------------------------------

\section{Lefschetz numbers for specific involutions} \label{section_LNSI}

 Let $\sigma$ be the complex conjugation. Its action on $\H$ is defined by $ (z,r)\mapsto (\bar{z},r)$. It also acts on the $\sl (\C)$ by acting on the entries of a matrix in the obvious way. If $M\in\sl (\C)$, the we write ${}^\sigma M$ or simply $\bar{M}$ for the image of $M$ under the action of $\sigma$. 

 Below, we will also consider the \emph{twisted complex conjugation}, which will be denoted by $\tau$. It acts on $\H$ via $(z,r)\to (-\bar{z},r)$ where $\bar{z}$ denotes the complex conjugate of $z$.  Its action on $\sl (\C)$ is defined as follows $$(\begin{smallmatrix} a & b \\ c & d \end{smallmatrix}) \mapsto (\begin{smallmatrix} \bar{a} & -\bar{b} \\ -\bar{c} & \bar{d} \end{smallmatrix}) $$ where the bar in the notation denotes the complex conjugation. It is convenient to regard $\tau$ as the composition $\alpha \circ \sigma = \sigma \circ \alpha$ where $\alpha(\begin{smallmatrix} a&b \\ c&d \end{smallmatrix} ) = (\begin{smallmatrix} a& -b \\ -c&d \end{smallmatrix})=
\beta (\begin{smallmatrix} a&b \\ c&d \end{smallmatrix}) \beta$ where $\beta :=( \begin{smallmatrix} -1& 0 \\ 0&1 \end{smallmatrix})$,  for 
every $(\begin{smallmatrix} a&b \\ c&d \end{smallmatrix}) \in \sl(\C)$ and $\alpha (z,r) = (-z,r)$ for every $(z,r) \in \mathbb{H}$. 

Both $\sigma$ and $\tau$ are orientation-reversing and they can be extended to the Borel-Serre compactification naturally (see \cite{rohlfs-84} Section 1.4).  The action of $\sigma$ on $E_{k,k}$ can be described as follows: $\sigma(P \otimes Q)= Q \otimes P$. Similarly, we have $\tau(P \otimes Q)=
(\begin{smallmatrix} -1 & 0 \\ 0&1 \end{smallmatrix}) Q \otimes (\begin{smallmatrix} -1 & 0 \\ 0&1 \end{smallmatrix}) P$. 
These actions are compatible with those on $\sl(\C)$.
 
 In this section, we discuss the Lefschetz numbers for these two involutions. We will use the symbol $\rho$ when we want 
to state results which are true for both of them.

%----------------------------------------------------------------

\subsection{Lefschetz numbers of $\sigma$ for principal congruence subgroups}

Let $\Gamma =\Gamma(N) \subseteq \sl (\O)$ be a principal congruence subgroup of level $(N) \triangleleft \O$.  Denote its image in $\psl (\O)$ by $\bar{\Gamma} $.  Then for $N>2$, $\bar{\Gamma}$ is torsion-free.  Let $Y_\Gamma=\bar{\Gamma}\backslash \H$. 
 
 In this section, we will use Theorem \ref{thm_lef} to calculate the Lefschetz numbers $L(\sigma,\Gamma(N),E_{k,k})$.  In order to do this, we need to understand the ``decomposition" of the fixed point set $Y_\Gamma^\sigma$, and this is  done by Rohlfs in Section 4.1. of \cite{rohlfs-78}. 
 
  Let $\gamma_1=(\begin{smallmatrix} 1 & 0 \\ 0 & 1\end{smallmatrix})$, $\gamma'_1=(\begin{smallmatrix} 1 & \sqrt{d} \\ 0 & 1\end{smallmatrix})$ and $\gamma_2=(\begin{smallmatrix} 0 & -1 \\ 1& 0\end{smallmatrix})$. Let $\gamma_2'$ be $(\begin{smallmatrix} 1 +\sqrt{d} & (2-d)/2 \\ -2 & -1 + \sqrt{d}\end{smallmatrix})$ if $d\equiv 2 \mod 4$, and  $(\begin{smallmatrix} \sqrt{d} & (d-1)/2 \\ 2 & -1 +\sqrt{d}\end{smallmatrix})$ if $d\equiv 1 \mod 4$. Notice that $\gamma_1,\gamma'_1 \in H(\sigma,1)$ and $\gamma_2,\gamma'_2 \in H(\sigma,2)$, where $H(\sigma,i)$, for $i=1,2$, is the generalized Galois cohomology set as defined above.

 Note also that, since $\bar{\Gamma}$ is torsion-free, $H(i)=\emptyset$ for $i>2$. As described in Section \ref{sec_ltffree},  the fixed point set $Y_\Gamma^\sigma$ is a union of surfaces and points parametrized by the cohomology classes in $H^1(\sigma,\bar{\Gamma})$. If $\Gamma$ does not contain $-1$, then we can identify $H^1(\sigma,\bar{\Gamma})$ with $H(1)$, and if $\Gamma$ contains $-1$, then we can identify it with $H(1)\cup H(2)$.  
 
  The locally symmetric space $F(\gamma)$, defined in section \ref{sec_ltffree}, is a surface if $\gamma\in H(1)$ and is a point if $\gamma\in H(2)$. In \cite{rohlfs-78}, Rohlfs gives the number of translations of the surfaces corresponding to $\gamma_1,\gamma'_1$, and the number of translations of the points corresponding to $\gamma_2,\gamma'_2$. 
  
  \begin{theorem}[Rohlfs, Theorem 4.1. of \cite{rohlfs-78}]\label{thm_table}
 Let $D$ be the discriminant of $K/\Q$ and $t$ be the number of distinct prime divisors of $D$.  Let $(N)=\prod_{p|D} \p_p^{j_p} \prod_{p\nmid D} (p)^{j_p}$ be an ideal with $N>2$, and let $\Gamma=\Gamma(N)$ be the principal  congruence subgroup of  level $(N)$.  Let $s=|\{p\mid p|D, p\neq 2 \text{ and } j_{p}\neq 0\}|$. 

Then $Y_\Gamma$ consists of only the translations of surfaces $F(\gamma_1)$ and $F(\gamma'_1)$ and the number of translations of these surfaces are denoted by $A$ and $B$ respectively in the table below.

\begin{center}   
{\large \begin{tabular}{ | l  | l | c | r | }
\hline
 $d$ &  $j_2$ & $A$  &  $B$ \\ \hline
 
 $d\equiv 1 (4)$ & $\geq 0$ & $2^{t-s}$ & $0$\\ \hline
 
$d\equiv 2 (4)$ & $0$ & $2^{t-s}$ & $2^{t-s-1}$\\     

                            & $1$ & $2^{t-s}$ & $2^{t-s-1}$ \\
                            
                            & $2$ & $8\cdot 2^{t-s}$ & $0$\\
                            
                            & $\geq 3$ & $8\cdot 2^{t-s-1}$ & $0$\\ \hline
                            
$d\equiv 3 (4)$ & $0$ & $2^{t-s}$ & $2^{t-s-1}$\\     

                            & $1$ & $2^{t-s}$ & $0$ \\
                            
                            & $2$ & $8\cdot 2^{t-s}$ & $0$\\
                            
                            & $j_2=2n+1\geq 3$ & $2^{t-s-1}$ & $0$\\ 
                            
                            & $j_2=2n\geq 4$ & $8\cdot 2^{t-s-1}$ & $0$\\  \hline

\end{tabular}}
\end{center}

 \end{theorem}

 Now, using Theorem \ref{thm_lef} and Theorem \ref{thm_table}, we want to calculate the Lefschetz number for $\Gamma(N)$ for the above case.
 
 \begin{proposition}\label{prop_lefprin}
 Let $\Gamma(N),A,B$ be as in the above theorem. Then 
$$L(\sigma,\Gamma(N),E_{k,k})=(A + 2B)   \dfrac{-N^3}{12} \prod_{p|N} (1-p^{-2}) \cdot (k+1).$$
 \end{proposition}
 
 \begin{proof}
 For each $\gamma\in H(1)$, by Lemma \ref{lem_trreform}, $\tr(\gamma\sigma \mid E_{k,k})=\tr(1 \mid E_{k,k})=(k+1)$. Therefore, by Theorem \ref{thm_lef}, we just need to calculate the Euler-Poincare characteristics $\chi(\Gamma^{\gamma\sigma})$ for $\gamma_1$ and $\gamma'_1$.
 
  An easy calculation shows that $\Gamma^{\gamma_1\sigma}=\Gamma_N$, the principal congruence subgroup of $\sl (\Z)$ of level $N$. Let $Y_N$ denote the surface associated to $\Gamma_N$. It is well-known that $Y_N$ has 
$\frac{1}{2}N^2\prod_{p|N} (1-p^{-2})$ cusps. If $X_N$ denotes 
the compact surface obtained from $Y_N$ by adding the cusps, then by \cite[1.6.4]{shimura}, we have $\chi(X_N)=(-1/12)N^2(N-6)\prod_{p|N} (1-p^{-2})$. 
Therefore $\chi(\Gamma_N)=\chi(Y_N)=\chi(X_N)- \#\{ \textrm{cusps of }Y_N\}=(-1/12)N^3 \prod_{p|N} (1-p^{-2})$. 
  
  Let $h=(\begin{smallmatrix} 2 & \sqrt{d} \\ 0 & 1 \end{smallmatrix}$). Let $\Gamma_0(2N)$ denote the $\Gamma_0$-type congruence subgroup of $\sl(\Z)$ level $2N$. And, let $\Gamma_2:=\{(\begin{smallmatrix} x & y \\ z & t \end{smallmatrix}) \in \sl(\Z)\mid  x\equiv t \ (2)\}$. One can check that $h\Gamma^{\gamma'_1\sigma}h^{-1}=\Gamma_N\cap \Gamma_0(2N)$ if $d\equiv 2 (4)$, and $h\Gamma^{\gamma'_1\sigma}h^{-1}=\Gamma_N\cap \Gamma_2$ if $d\equiv 3 \ (4)$. Each of these groups has index $2$ in $\Gamma_N$. Therefore, $\chi(\Gamma^{\gamma'_1\sigma})=2\chi(\Gamma_N)$. Using the formula $\chi(\Gamma_N)$, we get the formula above.
  \end{proof}

\begin{corollary} \label{cor_lefprin}
 Let $p$ be an odd rational prime that is unramified over $K$. Then, for $n>0$ we have
$$L(\sigma,\Gamma(p^n),E_{k,k})=\begin{cases} 
-2^{t-1} \cdot \dfrac{p^{3n}-p^{3n-2}}{12} \cdot (k+1) 
\ \ \ \ \ \textrm{if} \ \ d \equiv 1 \mod 4  \\  \\
-2^{t} \cdot \dfrac{p^{3n}-p^{3n-2}}{12} \cdot (k+1)
\ \ \ \ \ \ \textrm{else.}\end{cases}$$
\end{corollary}

\begin{proof} Since $p$ is odd, $j_2=0$. Moreover since $p$ is the only divisor of the level, we have $s=1$. Thus in the 
case $d \equiv 1 \mod 4$, we have $A+2B=2^{t-1}$, where as in the other case $A+2B=2^t$. 
\end{proof}

%-----------------------------------------------------------------------------------------------------------

\subsection{{Lefschetz numbers for the full Bianchi groups}}

Let $\Gamma$ denote the full Bianchi group $\sl(\O)$. For $k=0$, that is $E_{k,k}=\C$, the Lefschetz numbers for $\sigma$ and 
$\tau$ were computed by Kr\"amer. For general $E_{k,k}$, these numbers were computed by Blume-Neinhaus.

For a rational prime $p$ which ramifies in $K$ and an integer $a$, let $(a | p)$ denote the Hilbert symbol. By definition, $(a|p)$ is equal $1$ if there is an element in some finite extension of $K_p$, the completion of $K$ at the unique prime ideal over $p$, whose norm is equal to $a$, and is equal to $-1$ otherwise. 
Equivalently, $(a | p)$ is the value at $a$ of the quadratic character associated to the local extension $\Q_p(\sqrt{d})/ \Q_p$. Note that if $p \not= 2$, then $(a|p)$ is equal to the
 Legendre symbol $\left ( \frac{a}{p} \right )$. 

\begin{theorem}[Blume-Nienhaus, \cite{blume}]\label{thm_leftau} 
Let $D$ be the discriminant of $K/\Q$ with $D_2$ its 2-part. Let $\rho$ represent either $\tau$ or $\sigma$. Also, put $q=1,-1$ depending on 
whether $\rho=\tau$ and $\rho=\sigma$ respectively.
\begin{eqnarray*}
(-1)^k L(\rho,\Gamma,E_{k,k})&=&  
  \frac{-q}{12}\prod_{\substack{p|D\\ p\neq 2}} (p+ \left ( \frac{q}{p} \right )) \prod_{\substack{p|D\\ p=2}} (D_2 + (q |2)) \cdot (k+1) \\
  &+&\frac{q}{12}  \prod_{\substack{p|D\\ p\neq 2}}  (1+ \left (\frac{-q}{p} \right )) \prod_{\substack{p|D\\ p=2}} (4 + (-q |2))  \cdot (-1)^k   (k+1) \\
                                          &+& \frac{1}{2} \prod_{\substack{p|D\\ p\neq 2}} (1+ \left (\frac{-2q}{p} \right )) \cdot \left (\frac{k+1}{4} \right ) \\
                                          &+& \frac{1}{3}\Biggl (  \prod_{\substack{p|D\\ p\neq 3}} (1+ (-3q |p) ) + (-1)^k \prod_{p|D} (1 + (-q |p))  \Biggr) \cdot \left (\frac{k+1}{3} \right ).
\end{eqnarray*}
Here products over empty sets are understood to be equal to $1$.
\end{theorem}

\begin{proof} Observe that in Blume-Nienhaus' notation, $\Gamma(1)=\Gamma(-1)=\sl(\O)$. For $q=1,-1$ respectively, 
his involutions $(\begin{smallmatrix} 0 & 1 \\ q & 0 \end{smallmatrix}) \sigma$ (see Theorem V.5.3. of \cite{blume}) differ from our 
$\tau=(\begin{smallmatrix} -1 & 0 \\ 0 & 1 \end{smallmatrix}) \sigma$ and $\sigma$ by $(\begin{smallmatrix} 0 & 1 \\ -1 & 0 \end{smallmatrix})$. However at the level of cohomology, his involutions induce the same action as ours: conjugation by $(\begin{smallmatrix} 0 & 1 \\ -1 & 0 \end{smallmatrix})$ is an inner-automorphism of $\sl(\O)$ and hence induces trivial action on the cohomology, see \cite[p.79]{brown}.
\end{proof}

%-------------------------------------------------------------------------------------------------------------
%-------------------------------------------------------------------------------------------------------------
%-------------------------------------------------------------------------------------------------------------

\section{Trace on the Eisenstein cohomology} \label{section_eisenstein}

In order to obtain lower bounds for the cuspidal cohomology using the Lefschetz numbers, we have seen that one needs to compute the 
trace of the involution on the Eisenstein part of the cohomology in all degrees. This is no simple task. 

In this section we study the trace of involutions induced by $\sigma$ and $\tau$ on the Eisenstein part of the cohomology. The boundary 
$\partial X_{\Gamma}$ is a disjoint union of 2-tori, each closing a cusp of $Y_{\Gamma}$. 
The set of cusps of $\Gamma$ can be identified with the orbit space $\Gamma \backslash \mathbb{P}^1(K)$. 
It is well-known that the number of cusps is $h(K)$, the class number of $K$, when $\Gamma$ is the full Bianchi group.

The fundamental group of a 2-torus is a free abelian group on two generators and it is easy to compute the 
size of its cohomology. 
\begin{proposition} Let $\Gamma$ be a congruence subgroup of a Bianchi group and let $c(\Gamma)$ denote the number of cusps of $\Gamma$. Then 
$$\dim H^0(\partial X_{\Gamma},\bar{\E}_k) =dim H^2(\partial X_{\Gamma},\bar{\E}_k)= c(\Gamma)$$
$$\dim H^1(\partial X_{\Gamma},\bar{\E}_k) = 2 \cdot c(\Gamma).$$
\end{proposition} 

The long exact sequence associated to the pair $(X_{\Gamma}, \partial X_{\Gamma})$ is compatible with 
the action of the involution $\tau$. 
It follows from algebraic topology that for $k>0$, the image of the restriction map 
$$H^i(X_{\Gamma},\bar{\E}_k) \rightarrow H^i(\partial X_{\Gamma},\bar{\E}_k)$$ 
is onto when $i=2$ and its image has half the rank of the target space when $i=1$. Hence we have the following. 

\begin{corollary}\label{cor_eisensteinbound}
Let $k > 0$ and $\Gamma$ as above. Then 
$$ \dim H^i_{Eis}(\Gamma,E_{k,k}) = c(\Gamma),  \ \ \  i=0,1,2.$$
Hence 
$$|\tr(\tau^i_{Eis}) | \leq c(\Gamma)$$
for any involution $\tau$. 
\end{corollary}

The following is direct consequence a result of Serre (see \cite[Th\`eor\'eme 9]{serre}).
\begin{proposition} Let $\Gamma=\sl(\O)$. Then the image of the restriction map
$$H^1(X_{\Gamma},\C) \rightarrow H^1(\partial X_{\Gamma},\C)$$ 
is inside the $-1$-eigenspace of complex conjugation acting on $H^1(\partial X_{\Gamma},\C)$.
\end{proposition}

Let us note that this result is extended to all maximal orders of $M_2(K)$  (with complex conjugation twisted accordingly) by Blume-Nienhaus \cite[V.5.7.]{blume} and by Berger \cite[Section 5.2.]{berger}.

\begin{corollary} \label{cor_serre} Let $\sigma^i_{Eis}$ be the involution on $H^i_{Eis}(\sl(\O), \C)$ given by complex conjugation. Then
$$\tr(\sigma^0_{Eis})=1, \ \ \ \  \tr(\sigma^1_{Eis})=-h(K), \ \ \ \ \tr(\sigma^2_{Eis})=-2^{t-1}+1$$
where  $t$ be the number of primes that ramify in $K$ and $h(K)$ is the class number of $K$.
\end{corollary}
\begin{proof} For convenience, put $X=X_{\sl(\O)}$. The claim for $\sigma^0_{Eis}$ follows from the fact  that $H^0_{Eis}(X,\C) = H^0(X,\C)=\C$. The action of $\sigma$ on the latter is trivial. The claim for $\sigma^1_{Eis}$ follows immediately from Serre's result above. It is well-known that the set of cusps of $\sl(\O)$ is in bijection with the class group of $K$ and the action of complex conjugation $\sigma$ on the cusps translates to taking inverse in the class group. Hence  an element of the class group is fixed by $\sigma$ if it is of order $2$. Genus Theory tells us that the number of elements of order 2 in the class group is $2^{t-1}$, implying that the trace of the involution induced by $\sigma$ on $H^0(\partial X,\C)$ is $2^{t-1}$. See \cite[Section 9]{serre} for more details. It follows from Poincar\'e duality and the fact that complex conjugation is orientation-reversing that the trace of the involution induced by $\sigma$ on $H^2(\partial X,\C)$ is $-2^{t-1}$. The long exact sequence associated to the pair $(X, \partial X)$ tells us that that $H^2(\partial X,\C) \simeq H^2_{Eis}( X,\C) \oplus  H^3(X, \partial X,\C)$. Here the last summand is isomorphic to $\C$ and $\sigma$ acts on it as $-1$, which follows from the fact that the action of $\sigma$ on $H^0(X,\C)$ is trivial. This gives the claim for $\sigma^2_{Eis}$. 
\end{proof}

We will now consider the case of principal congruence subgroups. For a positive integer $N$, let us begin by reminding the reader 
the formula for the number of cusps of the principal congruence subgroup $\Gamma(N)$ of level $(N)$:
\begin{equation} \label{cuspformula}
c(\Gamma(N)) = h(K) \cdot \dfrac{\# \sl(\O / (N))}{N^2} = h(K) \cdot N^4 \prod_{\p} (1- \textbf{N}(\p)^{-2}).
\end{equation}
where product on the right hand side runs over the prime factors $\p$ of the ideal $(N)$ and $\textbf{N}(\p)$ denotes the norm of $\p$.

The following theorem generalizes part of the above Corollary and part of the results announced by Harder at the very end of \cite{harder-75}.

\begin{theorem} \label{trace2_eisenstein}
 Let $K$ be an imaginary quadratic field and $t$ be the number of rational  primes ramifying in. Let $N=p_1^{n_1} \hdots p_r^{n_r}$ be a positive number whose prime divisors $p_i$ are unramified in $K$ and let $\Gamma(N)$ be the principal congruence subgroup of the Bianchi group $\sl(\O)$ of level (N). Then
$$\tr(\sigma \mid H^2_{Eis}(\Gamma(N),E_{k,k}))=-2^{t-1} \cdot \prod_{i=1}^r (p_i^{2n_i}-p_i^{2n_i-2}) +\delta(0,k),$$ 
where $\delta$ is the Kronecker $\delta$-function, in other words, $\delta(0,k)=0$ unless $k=0$ in which case $\delta(0,k)=1$. In particular, the trace of $\sigma^2_{Eis}$ on $H^2(\sl(\O),E_{k,k})$ is 
$$-2^{t-1}+\delta(0,k).$$
\end{theorem}

\begin{proof} For convenience, let $G$ denote the Bianchi group $\sl(\O)$. Assume until the very end of the proof that $k>0$.
By the results of Section \ref{SL}, it suffices to compute trace of $\sigma^2$ on $H^2_{Eis}(G, \C[\Gamma \backslash G] \otimes E_{k,k})$. Let $\mathcal{M}_k$ be the locally constant sheaf on $X_{G}$ induced from $\C[\Gamma \backslash G] \otimes E_{k,k}$. 
As the restriction map $H^2(X_{G}, \mathcal{M}_k) \rightarrow H^2(\partial X_{G}, \mathcal{M}_k)$ is onto (here we use that $k>0$), it suffices to compute the trace of $\sigma^2$ on $H^2(\partial X_{G}, \mathcal{M}_k)$. As before, Poincar\'e duality 
together with the fact that $\sigma$ reverses the orientation reduce the problem to computing the trace of  $\sigma^0$ on 
$H^0(\partial X_G, \mathcal{M}_k)$ instead. 

The cohomology of the boundary can be expressed as a direct sum of the cohomology of the boundary components, which are 2-tori; 
\begin{equation} \label{cuspdec}
H^0(\partial X_G, \mathcal{M}_k) \simeq \bigoplus_{c} H^0(U_c, \C[\Gamma \backslash G] \otimes E_{k,k})
\end{equation}
where the summation runs over the cusps of $G$ and $U_c$ is the unipotent radical of the stabilizer 
of the cusp $c$. The action of $\sigma^0$ on the left hand side of (\ref{cuspdec}) translates to an action on the right hand side. This 
the action stems from that of complex conjugation $\sigma$ on the cusps, which amounts to inversion in the class group of $K$. If $c$ is a cusp, then $\sigma$ takes $U_c$ to $U_{\sigma(c)}$. If $c \not = \sigma(c)$, then  
$$H^0(U_c,\C[\Gamma \backslash G] \otimes E_{k,k}) \oplus H^0(U_{\sigma(c)},\C[\Gamma \backslash G] \otimes E_{k,k})$$
is a $\sigma^0$-invariant subspace of the right hand side of (\ref{cuspdec}). As $\sigma^0$ takes the basis of the first summand to the basis of the second summand, the trace of $\sigma^0$ on this subspace is 0. Hence we need 
to consider cusps $c$ which are fixed by complex conjugation $\sigma$. As mentioned before, there are $2^{t-1}$ of these.

For the rest of the proof, assume that $c$ is a cusp fixed by complex conjugation $\sigma$. Our goal is to compute the trace of $\sigma^0_{Eis}$ on $H^0(U_c, \C[\Gamma \backslash G] \otimes E_{k,k})$. Recall the action of complex conjugation on the $G$-module $\C[\Gamma \backslash G] \otimes E_{k,k}$ from Section \ref{SL}. As $H^0$ is taking invariants, we have 
$$H^0(U_c, \C[\Gamma \backslash G] \otimes E_{k,k}) \simeq H^0(U_c, \C[\Gamma \backslash G]) \otimes H^0(U_c, E_{k,k}).$$ 
This isomorphism is $\sigma^0$-equivariant. 

We first prove that, without loss of generality, we can take $c$ to be the cusp $\infty$ at infinity. To see this, assume that $c$ is given by $\frac{x}{y}$ with $x,y \in \O$. Then there is a matrix $A=(\begin{smallmatrix} x & s \\ y & t \end{smallmatrix}) \in \sl(K)$ such that $\infty \cdot A=c$. As $c$ is fixed by $\sigma$, 
it corresponds to an element in the ideal class group of $K$ which is of order 2. Work of Smillie and Vogtmann \cite{smilie-vogtmann} shows us that conjugation by $A$ gives rise to an orientation-preserving involutary automorphism $\phi_A$ of $Y_G$ which {\em commutes} with the orientation-reversing involutary automorphism 
of $Y_G$ induced by $\sigma$. The automorphism $\phi_A$ extends to an automorphism of $X_G$, inducing an isomorphism (still denoted $\phi_A$)
between the boundary component 2-torus at the cusp $\infty$ and the boundary component 2-torus at the cusp $c$. The isomorphism $\phi_A$ commutes with the automorphisms induced by $\sigma$ on these two boundary components. It follows that the trace of $\sigma$ on $H^0(U_{\infty},\C[\Gamma \backslash G] \otimes E_{k,k})$ is equal to its trace on $H^0(U_c,\C[\Gamma \backslash G] \otimes E_{k,k})$.

So let us assume that $c=\infty$ below.
It is easy to show that $E_{k,k}^{U_c}$ is a one-dimensional space that is generated by $X^k \otimes \bar{X}^k$. 
Clearly $\sigma$ acts as identity on this space. Therefore the trace of $\sigma_{Eis}^0$ on $H^0(U_c, E_{k,k})$ is $1$. 
Our discussion so far already shows that the trace of $\sigma_{Eis}^2$ on $H^2_{Eis}(G,E_{k,k})$ is equal to  $-2^{t-1}$.  

Now let us move on to deal with $H^0(U_c, \C[\Gamma \backslash G])$. By the above paragraph, we see that the dimension of 
$H^0(U_c,\C[\Gamma \backslash G])$ is equal to $c(\Gamma) / h$ where $c(\Gamma)$ is the number of cusps of $\Gamma$ and $h$ is the class number of $K$. 

We will first work in the {\em special situation} where $N=p^n$ with $p$ a rational prime that is unramified in $K$. For convenience put  $\Gamma=\Gamma(p^n)$ and $R=\O / (p^n)$. The coset space $\Gamma \backslash G$ can be identified with the finite group $\sl(R)$. Let $U_c(R)$ denote the image of $U_c$ inside $\sl(R)$ under the reduction modulo $(p^n)$ map. The action of $U_c$ on $\sl(R)$ is the right regular action of $U_c(R)$ on  $\sl(R)$. Since our $p$ is unramified in $K$, we can choose $x,y \in \O$ prime to $p$ and thus $\bar{c}:=(\bar{x}:\bar{y})$ gives an element of the projective line $\P^1(R)$ over the ring $R$, for a definition see \cite[p.281]{cremona}. We have $U_c(R)=U_{\bar{c}}$, the unipotent radical of the stabilizer of $\bar{c}$ in 
$\sl(R)$.

First assume that $p$ {\em splits} in $K$, that is $(p)=\p \bar{\p}$. Clearly $\sigma$ acts on 
$$\sl(R) \simeq \sl(\O/ \p^n) \times \sl(\O/ \bar{\p}^n)$$
by swapping the coordinates, that is, 
$\sigma (X,Y) = (Y,X)$ for every $(X,Y) \in \sl(\O/ \p^n) \times \sl(\O/ \bar{\p}^n))$. Observe that, in the obvious notation, we have 
$$\C[\sl(R)]^{U_c(R)} \simeq \C[\sl(\O/ \p^n)]^{U_c(\O/ \p^n)} \otimes \C[\sl(\O/ \bar{\p}^n)]^{U_c(\O/ \bar{\p}^n)}.$$

Let us identify $\O / \p^n$ and $\O / \bar{\p}^n$ with $\Z / p^n \Z$ and simply consider $\C[\sl(\Z / p^n \Z)]^{U_c(\Z / p^n \Z)}$. Put 
$$\theta= \sum_{u \in U_c(\Z / p^n \Z)} u \in \C[\sl(\Z / p^n \Z)].$$ 
Then $\theta$ is clearly invariant under the right regular action of $U_c(\Z / p^n \Z)$. Fix a set $S$ of representatives for cosets of 
$U_c(\Z / p^n \Z)$ in $\sl(\Z / p^n \Z)$. Then for every $s \in S$, the element 
$$\sum_{u \in U_c(\Z / p^n \Z)} su = s\theta \in \C[\sl(\Z / p^n \Z)]$$ 
is also fixed under the action of $U_c(\Z / p^n \Z)$. Dimension considerations show that the set $\mathbb{S}=\{ s \theta \mid s \in S \}$ forms a basis of $\C[\sl(\Z / p^n \Z)]^{U_c(\Z / p^n \Z)}$.  Thus 
the set $\mathbb{S} \times \mathbb{S}$ gives a basis of $\C[\sl(R)]^{U_c}$. Note that this basis is fixed by $\sigma$.

As $\sigma$ stabilizes the basis $\mathbb{S} \times \mathbb{S}$, it follows that the trace of the involution $\sigma^0$ on 
$\C[\sl(R)]^{U_c}$ is equal to the number of elements in $\mathbb{S} \times \mathbb{S}$ which are fixed under $\sigma$ 
(the same principle was in order in the proof of Corollary \ref{cor_serre} as well). Since $\sigma$ acts as  swapping 
coordinates, the set of $\sigma$-fixed elements is the diagonal which is of cardinality 
$\# \mathbb{S}= [\sl(\Z / p^n \Z) : U_c(\Z / p^n \Z)].$
To compute the latter, observe that 
$$[\sl(\Z / p^n \Z) : U_c(\Z / p^n \Z)] = \# \P^1(\Z / p^n \Z) \cdot \#(\Z / p^n \Z)^*=(p^n+p^{n-1})(p^n-p^{n-1})=p^{2n}-p^{2(n-1)}$$
where $(\Z / p^n \Z)^*$ is the group of units in $\Z / p^n \Z$. 
This follows from an investigation of the subgroups $U_c(\Z/ p^n \Z) = U_{\bar{c}} \subset B_{\bar{c}} \subset \sl(\Z / p^n \Z)$ 
where $B_{\bar{c}}$ is the stabilizer of $\bar{c}=(\bar{x}: \bar{y}) \in \P^1(\Z / p^n \Z).$

Let us now deal with the case where $p$ stays inert in $K$, that is $(p)=\p$ is a prime ideal (of norm $p^2$). Computing 
a basis of $H^0(U_c,\C[\sl(R)])$ goes along the above lines. Put 
$$\theta= \sum_{u \in U_c(R)} u \in \C[\sl(R)].$$ 
Fix a set $S$ of representatives for cosets of $U_c(R)$ in $\sl(R)$. 
Then the set $\mathbb{S}=\{ s \theta \mid s \in S \}$ forms a basis of $\C[\sl(R)]^{U(R)}$. 

The action of $\sigma$ on $\sl(R)$ is entry-wise, that is, given a matrix in $\sl(R)$, $\sigma$ acts on its entries. So we need to consider the $\sigma$-action on $R$. The easiest way to make this action concrete is to work with a set of representatives 
$T$ in $\O$ which biject onto $R$ when reduced modulo $(p)^n$. In our case, such a set $T$ is given by 
$$\{ a+b \cdot \omega \mid 0 \leq a,b \leq p^n-1 \} \subset \O$$
where $\omega$ is the standard generator $\O = \Z + \Z \omega$. The action of $\sigma$ on $R$ descends from the action of $\sigma$ 
on $T$: $a+b \cdot \omega \mapsto a + b \cdot {}^{\sigma}\omega$. As above, we want to compute the number of elements in 
$\mathbb{S}$ which are fixed under $\sigma$. 

As $U_c$ is stable under $\sigma$ (recall that $c=\infty$), its reduction $U_c(R)$ is stable under the action of $\sigma$ on $\sl(R)$, hence $\sigma$ acts on the coset 
space of $U(R)$ in $\sl(R)$. As the element $\theta$ is fixed under $\sigma$, we need to compute the number of $s$ in (some fixed) $S$ 
which are fixed by $\sigma$. This is the same as counting the $\sigma$-fixed elements in the coset space. 
As implicitly used above, the coset space of $U_c(R)$ in $\sl(R)$ can be ``identified" with $\P^1(R) \times R^*$. 
As an indication, let us point out that $\P^1(R)$ is in bijection with the coset space of a Borel 
subgroup $B_c(R)$ in $\sl(R)$ and $R^*$ is in bijection with the coset space of $U_c(R)$ in $B_c(R)$. Under this 
identification, the action of $\sigma$ translates to the natural action $\sigma(x:y)=({}^{\sigma}x : {}^{\sigma}y)$ 
on $\P^1(R)$ and it gives the usual action on $R^*$. Now the $\sigma$-fixed elements of the coset space 
of $U_c(R)$ in $\sl(R)$ can be identified with the product $\P^1(R)^{\sigma} \times (R^*)^{\sigma}$ 
of $\sigma$-fixed elements of $P^1(R)$ and $R^*$. One can check that $\P^1(R)^{\sigma}=\P^1(R^{\sigma}) \simeq \P^1(\Z / p^n \Z)$ 
which has cardinality $p^n+p^{n-1}$. The set $R^*$ can be identified with $\{ a+b \cdot \omega \in T \mid p\nmid a \ \& \ p \nmid b \}$. 
Hence $(R^*)^{\sigma}$ is given by $\{ a+b \cdot \omega \in T \mid p \nmid a, \ \  b=0 \}$ which is of cardinality $p^n-p^{n-1}$. 
Multiplying these two cardinalities gives us the desired quantity $p^{2n}-p^{2(n-1)}$.
 
Now let us assume that $N=p_1^{n_1} \hdots p_r^{n_r}$ is positive number whose prime divisors $p_i$ are unramified in $K$. The general result follows from the simple fact that  
$$\sl(\O / (N)) \simeq \sl(\O / (p_1)^{n_1}) \times \hdots \sl(\O / (p_r)^{n_r}).$$

The case $k=0$ follows from the basic observations that were employed at the end of the proof of Corollary \ref{cor_serre}.
\end{proof}

Now let us compute the trace of $\tau^2_{Eis}$. Recall that $\tau$ can be regarded as the composition $\alpha \circ \sigma = \sigma \circ \alpha$ 
where $\alpha(\begin{smallmatrix} a&b \\ c&d \end{smallmatrix} )=
( \begin{smallmatrix} -1& 0 \\ 0&1 \end{smallmatrix}) (\begin{smallmatrix} a&b \\ c&d \end{smallmatrix})
( \begin{smallmatrix} -1& 0 \\ 0&1 \end{smallmatrix})$ for 
every $(\begin{smallmatrix} a&b \\ c&d \end{smallmatrix}) \in \sl(\C)$ and $\alpha (z,r) = (-z,r)$ for every $(z,r) \in \mathbb{H}$.  

\begin{theorem} \label{trace2_eisenstein_tau} 
 Let $K$ be an imaginary quadratic field and $t$ be the number of rational 
primes ramifying in. Let $N=p_1^{n_1} \hdots p_r^{n_r}$ be a positive number whose prime divisors $p_i$ are unramified in $K$ and 
let $\Gamma(N)$ be the principal congruence subgroup of the Bianchi group $\sl(\O)$ of level $(N)$. 
Then
$$\tr(\tau \mid H^2_{Eis}(\Gamma(N),E_{k,k}))=-2^{t-1} \cdot \prod_{i=1}^r (p_i^{2n_i-1}-p_i^{2n_i-2}) +\delta(0,k),$$
where $\delta$ is the Kronecker $\delta$-function, in other words, $\delta(0,k)=0$ unless $k=0$ in which case $\delta(0,k)=1$.
In particular, the trace of $\tau$ on $H^2_{Eis}(\sl(\O),E_{k,k})$ is
$$-2^{t-1}+\delta(0,k).$$
\end{theorem}

\begin{proof} We will follow the proof for $\sigma^2_{Eis}$ closely. As before, for convenience put $G=\sl(\O)$ and assume that $k>0$ 
until the very end. First, observe that $\alpha$ fixes the cusps of $G$, giving that the action of $\tau$ on the cusps of $G$ is the same as that of complex conjugation $\sigma$. 

As in the previous proof, we end up having to compute the trace of the action of $\tau^0$ on $H^0(U_c, \C[\Gamma \backslash G] \otimes E_{k,k})$ for cusps $c$ which are fixed under complex conjugation (and hence $\tau$). Using the same argument, we can (and we do) assume that $c=\infty$ below. 
As before, we have 
$$H^0(U_c, \C[\Gamma \backslash G] \otimes E_{k,k}) \simeq H^0(U_c, \C[\Gamma \backslash G]) \otimes H^0(U_c, E_{k,k}).$$ 
This isomorphism is $\tau^0$-equivarient. 

As $\tau$ sends $X^k \otimes \bar{X}^k$ to $(-X)^k \otimes (-\bar{X})^k = (-1)^{2k}(X^k \otimes \bar{X}^k) $, the trace of $\tau^0$ on $H^0(U_c, E_{k,k})$ is $1$. This shows that the trace of $\tau_{Eis}^2$ on $H^2_{Eis}(G,E_{k,k})$ is equal to  $-2^{t-1}$.  

Assume that $N=p^n$ with $p$ a rational prime that is unramified in $K$. 
For convenience put  $\Gamma=\Gamma(p^n)$ and $R=\O / (p)^n$. Identify $\Gamma \backslash G$ with $\sl(R)$. 

First assume that $p$ splits in $K$, that is $(p)=\p \bar{\p}$. The action of $\tau$ on  
$\sl(R)=\sl(\Z / p^n \Z) \times \sl(\Z / p^n \Z)$ is as follows, $\tau (X,Y) = (\beta Y\beta,\beta X \beta)$ for every 
$(X,Y) \in \sl(\Z / p^n \Z) \times \sl(\Z / p^n \Z)$.  In the proof of the previous theorem, we described a set $\overline{S}$ such that $\overline{S} \times \overline{S}$ forms a basis of the space $H^0(U_c,\C[\sl(R)])$. The set $\overline{S}$ can be identified with 
$\P^1(\Z / p^n \Z) \times (\Z / p^n \Z)^*$. One sees then that $\overline{S} \times \overline{S}$ is fixed by $\tau$ and thus  
to compute the desired trace, it is enough to compute the elements in this basis which are fixed by $\tau$. Such element are described as 
$(X,X) \in  \overline{S} \times \overline{S}$ with $X \in \overline{S}$ diagonal. The diagonal elements in $\overline{S}$ are in bijection with 
the set $\{ (0:*) \} \times (\Z / p^n \Z)$ which has cardinality $p^{n-1}(p^n-p^{n-1})$. 

Let us now consider the case where $p$ is inert in $K$. The action of $\tau$ on the coset space $\sl(R)$ is 
described as $\tau(A)=\beta \sigma(A) \beta$, where the action of $\sigma$ was given in the proof of the previous theorem. 
In the same proof, we computed a basis $\overline{S}$ of $H^0(U_c,\C[\sl(R)])$ for the case where $p$ is inert. 
The set $\overline{S}$ can be identified with the set $\P^1(R) \times (R)^*$ and we see that it is fixed by $\tau$. 
The elements of $\overline{S}$ which are fixed under $\tau$ are given by 
$\{ (0,*) \} \times (R^{\sigma})^* \subset \P^1(R^{\sigma}) \times (R^{\sigma})^*$ which 
has cardinality $(p^{n-1})(p^n-p^{n-1})=p^{2n-1}-p^{2n-2}$. 

The general cases follow from the same steps that were taken in the proof of the previous lemma.
\end{proof}

%-----------------------------------------------------------------
\subsection{Trace on $H^1_{Eis}$}

In this subsection, we will compute the trace of $\sigma$ on $H^1_{Eis}(\Gamma,\C)$. Our strategy is to use the explicit 1-cocycles defined by Sczech in \cite{sczech}
which produce a basis for $H^1_{Eis}(\Gamma(N),\C)$. 

Consider $\O$ as a lattice in $\C$. For $k=0,1,2$ and $u \in \C$ put 
$$E_k(u) = E_k(u,\O) = \sum'_{w \in \O} (w+u)^{-k} |w+u|^{-s} \mid_{s=0}$$
where $\hdots \mid_{s=0}$ means that the value is defined by analytic continuation to $s=0$. Moreover define $E(u)$ by
setting
$$2 E(u) = \begin{cases} 2 E_2(0), \ \ \ \ \ \ \ \ \ \  \ \ u \in \O \\ \wp(u)-E_1(u)^2, \ \ \ u \not \in \O \end{cases}$$
where $\wp(u)$ denotes the Weierstrass $\wp$-function. 

Let $N$ be a positive integer. Given $u,v \in \frac{1}{N}\O$, Sczech forms homomorphisms
$$\Psi(u,v) : \Gamma(N) \rightarrow \C$$
which depend only on the classes of $u$ and $v$ in $\frac{1}{N}\O / \O$. For 
$A=(\begin{smallmatrix} a & b \\ 0 & d \end{smallmatrix} ) \in \Gamma(N)$, we have the simple description 
$$\Psi(u,v)(A)= - \left ( \frac{\bar{b}}{d}\right ) E(u) - \frac{b}{d}E_0(u)E_2(v)$$
where 
$$ \left ( \frac{t}{s}\right ) = -1 + \# \{ y \mod s\O \mid y^2 \equiv t \mod s\O \}$$
is the Legendre symbol. For non-parabolic $A \in \Gamma(N)$ there is a similar but more complicated description which uses finite sums that involve the $E_k$'s, generalizing the classical Dedekind sums. 

It is shown by Sczech that the collection $\Psi(u,v)$ with $(u,v) \in (\frac{1}{N}\O / \O)^2$ live in the Eisenstein part of 
the cohomology and that the number of linearly independent such homomorphisms is equal to the number of cusps
of $\Gamma$. Thus they generate $H^1_{Eis}(\Gamma(N),\C)$. 

Ito showed in \cite{ito1} that, see also Weselmann \cite{weselmann}, up to a coboundary, the cocycles of Sczech are integrals of closed harmonic differential forms given 
by certain Eisenstein series defined on the hyperbolic 3-space $\mathbb{H}$. Following Ito, we can form an Eisenstein series $E_{(u,v)}(\tau,s)$ for $(\tau,s) \in \mathbb{H}\times \C$ with values in $\C^3$ associated to each cusp of $\Gamma(N)$.  As a function of $s$, $E_{(u,v)}(\tau,s)$ can be analytically continued to 
whole $\C$ and work of Harder \cite{harder-79} shows that differential 1-form on the hyperbolic 3-space induced by $E_{(u,v)}(\tau,s)$ is closed for
$s=0$. Ito showed that the cocycle given by the integral of this closed differential 1-form differs from the cocycle $\Psi(u,v)$ 
of Sczech by a coboundary. The fact that the above Eisenstein series associated to different cusps are linearly independent 
(they are non-vanishing only at their associated cusp) implies that the cohomology classes of Sczech cocycles which are associated to the cusps of $\Gamma(N)$ form a {\em basis} of $H^1_{Eis}(\Gamma(N),\C)$. 

In another paper \cite{ito2}, Ito provides us the following results:
$$\Psi(0,0)(\bar{A})=-\Psi(0,0)(A)$$
where bar means that we take the complex conjugates of the entries of the matrix $A$. More generally,  he proves that
$$\Psi(u,v)(\bar{A})= \dfrac{-1}{N^2}\sum_{s,t \in \frac{1}{N}\O / \O} 
\phi ( s\bar{v}-t\bar{u}) \Psi(s,t)(A)$$
where $\phi(z):=exp(2\pi i (z-\bar{z})/ D)$ with $D$ denoting the discriminant of $K$. 
Observe that when $(s,t)=(u,v)$ or $(s,t)=(0,0)$, we have $\phi(s\bar{v}-t\bar{u})=1$. 
Using this, let us write this summation in a more suggestive way: 
$$\Psi(u,v)(\bar{A})= \dfrac{-1}{N^2} \Bigg [ \bigg ( \sum_{\substack{s,t \in \frac{1}{N}\O / \O \\ (s,t) \not= (u,v) \\ (s,t) \not= (0,0)}} 
\phi( s\bar{v}-t\bar{u}) \Psi(s,t)(A)\bigg ) + \Psi(u,v)(A) + \Psi(0,0)(A) \Bigg ]$$

The latter formula sheds light onto the action of complex conjugation $\sigma$ on the Sczech cocycles which is given by
$$\sigma(\Psi(u,v))(A):=\Psi(u,v)(\bar{A}).$$
We see that $\sigma(\Psi(u,v))$ is expressed as summation over {\em all} the Sczech cocycles. We will 
regard $\sigma$ as a linear operator on the formal space $\C[\Psi_N]$ for which the Sczech cocycles are taken as basis. 

The pair $(0,0)$ in $(\frac{1}{N}\O / \O)^2$ never corresponds to a cusp of $\Gamma(N)$, so let us eliminate 
the term $\Psi(0,0)$ from the big summation. Using Ito's summation formula for the case $(u,v)=(0,0)$, we get 
$$\Psi(0,0)(\bar{A})=\dfrac{-1}{N^2} \Bigg [ \bigg ( \sum_{\substack{s,t \in \frac{1}{N}\O / \O \\ (s,t) \not= (0,0)}} 
\Psi(s,t)(A)\bigg )+ \Psi(0,0)(A) \Bigg ]$$
Now plug in the identity $\Psi(0,0)(\bar{A})=-\Psi(0,0)(A)$, we get 
$$\Psi(0,0)(A) = \dfrac{1}{N^2-1}\sum_{\substack{s,t \in \frac{1}{N}\O / \O \\ (s,t) \not= (0,0)}} \Psi(s,t)(A).$$
Now for $(u,v)\not=(0,0)$, we have
$$\Psi(u,v)(\bar{A})= \dfrac{-1}{N^2} \Bigg [ \bigg ( \sum_{\substack{s,t \in \frac{1}{N}\O / \O \\ (s,t) \not= (0,0)}} 
\phi( s\bar{v}-t\bar{u}) \Psi(s,t)(A)\bigg )+ \Psi(0,0)(A) \Bigg ].$$
Substitute the term $\Psi(0,0)(A)$, we get
$$\Psi(u,v)(\bar{A})=\dfrac{-1}{(N^2)(N^2-1)}\sum_{\substack{s,t \in \frac{1}{N}\O / \O \\ (s,t) \not= (0,0)}} \Psi(s,t)(A) + 
\dfrac{-1}{N^2} \sum_{\substack{s,t \in \frac{1}{N}\O / \O \\ (s,t) \not= (0,0)}} 
\phi( s\bar{v}-t\bar{u}) \Psi(s,t)(A).$$
Having eliminated $\Psi(0,0)$, we can regard $\sigma$ as a linear operator on the formal 
space $\C[\Psi^*]$ for which all Sczech cocyles except $\Psi(0,0)$ are taken as basis.
We see that the coefficient of the summand $\Psi(u,v)(A)$ on the right hand side of the equality is 
$$\frac{1}{(N^2)(N^2-1)}+\frac{-1}{N^2}=\frac{-1}{N^2-1}.$$
This implies that the trace of $\sigma$ on $\C[\Psi_N^*]$ is 
\begin{equation} \label{sigmatraceformula}
(N^4-1) \frac{-1}{N^2-1}= -(N^2+1).
\end{equation}

\vspace{.1 in}

Our goal is to apply the above results to the computation of the trace of $\sigma$ on $H^1_{Eis}(\Gamma(p^n),\C)$ 
for some rational prime $p$ which is unramified in $K$. 
We were able to do this only when $K$ is of class number one and $p$ is inert in $K$.  
In this case, the cusps of $\Gamma(p^n)$ are in bijection with the elements $(\bar{x},\bar{y})$ of order $p^{2n}$ in 
$(\O / (p^n))^2$ via the map $\frac{x}{y} \mapsto (y,-x)$. 
It follows from Equation (\ref{cuspformula}) that $c(\Gamma(p^n))=p^{4n}(1-p^{-4})=(p^{2n})^2-(p^{2n-2})^2$. 

In the rest of this subsection, we will prove the following result which is a partial generalization of a result announced by Harder in \cite{harder-75}.
\begin{theorem} \label{trace1_eisenstein}
Assume that $K$ is of class number one and let $p$ be a rational prime that is inert in $K$. Then we have 
$$ \tr(\sigma\mid H^1_{Eis}(\Gamma(p^n),\C)) = \begin{cases} -(p^2+1), \ \ \ \ \ \ \ \ \ \ \textrm{if} \ n=1 \\ 
                                                                              -(p^{2n}-p^{2n-2}), \ \ \ \textrm{if} \ n>1. \end{cases}$$ 
\end{theorem}

\begin{proof} We will proceed by induction. Let $n=1$. Then by a comparison with the number of cusps of $\Gamma(p)$, we see that 
the Sczech cocyles, excluding $\Psi(0,0)$, form a basis of $H^1_{Eis}(\Gamma(p),\C)$. Thus the trace if $\sigma$ on $\C[\Psi^*]$ 
is equal to the trace of $\sigma$ on $H^1_{Eis}(\Gamma(p),\C)$. By our observation above, we get the claim for $n=1$. 
 
Before we proceed with the inductive step, let us discuss the structure of cusps. The following diagram is commutative.
$$ \xymatrix{ \O/ (p) \ar[r]^{\varepsilon} & \O / (p^2) \ar[r]^{\varepsilon} & \O/ (p^3) \ar[r]^{\varepsilon} & \hdots \\ 
                   \frac{1}{p}\O / \O \ar[u] \ar[r]^{\varepsilon'} & \frac{1}{p^2}\O / \O \ar[u] \ar[r]^{\varepsilon'} & \frac{1}{p^3}\O / \O \ar[u] \ar[r]^{\varepsilon'} & \hdots}$$
The maps $\varepsilon$ are the natural inclusion maps $[x] \mapsto [px]$. the vertical arrows are the natural bijections 
that we mentioned above and the maps $\varepsilon'$ are induced by the natural inclusions $ \frac{1}{p}\O \subset  \frac{1}{p^2}\O$. 
The crucial observation is that the set of elements of order $p^{2n}$ in $(\O/ (p^n))^2$  is exactly $(\O/ (p^n))^2 \backslash (\varepsilon(\O / (p^{n-1})))^2$. 
Hence in order to find the trace of $\sigma$ on the Sczech cocyles which are associated to the cusps of $\Gamma(p^n)$, 
all we need to do is to compute the difference between the traces of $\sigma$ on $\C[\Psi_{p^n}]$ 
and $\C[\Psi_{p^{n-1}}]$. This is the same as the difference between the traces of $\sigma$ on $\C[\Psi_{p^n}^*]$ 
and $\C[\Psi_{p^{n-1}}^*]$ which we already computed in Equation \ref{sigmatraceformula}:
$$-(p^{2n}+1)-(-(p^{2n-2}+1))=-(p^{2n}-p^{2n-2})$$
as claimed.
\end{proof}

\begin{remark} 
\begin{enumerate} 
\item The above proof does not carry over to the case where $p$ is split. To see this, put $(p)=\p \bar{\p}$. Then 
the set of cusps of $\Gamma(p^n)$, which has cardinality $(p^{2n}(1-p^{-2}))^2$, is in bijection with the Cartesian 
product of the set of cusps of $\Gamma(\p^n)$ and the set of cusps 
of $\Gamma(\bar{\p}^n)$. Both of the latter sets are in bijection with the set of elements of order $p^n$ of $\Z / p^n \Z$. 
In a very similar way to the one in the proof, we have a commutative diagram 
$$ \xymatrix{ \hdots \ar[r] & (\O/\p^{n-1})^2 \times  (\O/\bar{\p}^{n-1})^2  \ar[r]^{\ \ \ \ \varepsilon_{\p} \times \varepsilon_{\bar{\p}} } &  
                   (\O/\p^{n})^2 \times  (\O/\bar{\p}^{n})^2 \ar[r] &  \hdots \\ 
                   \hdots \ar[r] & (\frac{1}{p^{n-1}}\O / \O)^2 \ar[u] \ar[r]^{\varepsilon'} & (\frac{1}{p^n}\O / \O)^2 \ar[u] \ar[r] & \hdots}$$
However, unlike in the inert case, the subset of elements in $ (\frac{1}{p^n}\O / \O)^2$ which correspond to the cusps of $\Gamma(p^n)$ is
{\em not} the complement of the image of $(\frac{1}{p^{n-1}}\O / \O)^2$ in $(\frac{1}{p^n}\O / \O)^2$. This obstructs the recursive use 
of Equation \ref{sigmatraceformula} in this case. In fact, the subset of elements in $ (\frac{1}{p^n}\O / \O)^2$ which correspond to the cusps 
of $\Gamma(p^n)$ is given by 
$$\left ( (\O/\p^{n})^2 \backslash \textrm{Im}(\varepsilon_{\p}) \right ) \times \left ( (\O/\bar{\p}^{n})^2 \backslash \textrm{Im}(\varepsilon_{\bar{\p}}) \right ). $$
However we do not see a way to isolate this set in a recursive way.

\item In order to treat $H^1_{Eis}(\Gamma(N), E_{k,k})$, 
vector-values versions of Sczech's cocycles should be developed or Harder's theory of Eisenstein cohomology \cite{harder_eiscoh} should be employed. Since both of the 
options require significant amount of extra work, we do not attempt any of them here.
\end{enumerate}
\end{remark}

These Eisenstein trace results together with Lefschetz number computations of previous sections can be plugged in the formula 
of Proposition \ref{lefschetzlowerbound}, giving {\em explicit} lower bounds for the cuspidal cohomology of Bianchi 
groups. We leave such tasks to the interested reader as the formulas will be quite complicated. In the special case of class number one $K$, 
weight $k=0$ and principle congruence subgroup $\Gamma(p)$ with $p$ inert in $K$, we checked that the explicit lower bound 
that we derive from Proposition \ref{lefschetzlowerbound} agrees with that given by Harder at the end of \cite{harder-75}.

%----------------------------------------------------------------------------------------------------------------------------------------------------------------------------
%---------------------------------------------------------------------------------------------------------------------------
%---------------------------------------------------------------------------------------------------------------------------

\section{Asymptotic lower bounds} \label{section_asymptotics}

Finding explicit formula for the dimension of $H^1_{cusp}(\Gamma,E_{k,k})$ for  
congruence subgroups $\Gamma$ of a Bianchi group is an important open 
problem in the theory. Recently there has been progress in understanding the asymptotic behaviour of the dimension. 

In the ``horizontal" direction, Finis, Grunewald and Tirao considered in \cite{fgt} the size of 
the cuspidal cohomology of a fixed congruence subgroup $\Gamma$ as the weight $E_{k,k}$ varied. They 
were able to increase the trivial asymptotic upper bound $k^2$ by a factor. In the case of $\Gamma=\sl(O)$, they provided a lower bound that is linear in $k$. A recent result of Marshall in \cite{marshall}, when applied to our situation,  improves the trivial asymptotic upper bound by a power.

\begin{theorem} Let $\Gamma$ be a congruence subgroup of a Bianchi group.
\begin{enumerate}
\item (Finis-Grunewald-Tirao \cite{fgt}) We have 
$$ k \ll \dim H^1_{cusp}(\Gamma, E_{k,k}) \ll \frac{k^2}{\log k}$$
as $k$ increases. The inequality on the left is proven only for the case $\Gamma=\sl(\O)$.
\item (Marshall \cite{marshall}) We have
$$ \dim H^1_{cusp}(\Gamma, E_{k,k}) \ll_{\epsilon} k^{5/3+\epsilon}$$
as $k$ increases. 
\end{enumerate}
\end{theorem}

In the ``vertical" direction, Calegari and Emerton considered in \cite{calegari-emerton} how the size 
of the cohomology, with fixed coefficient module,  varied in a tower of arithmetic groups. 
Their general result when applied to our situation gives the following.

\begin{theorem} (Calegari-Emerton \cite{calegari-emerton}) Let $\Gamma(\p^n)$ denote the 
principal congruence subgroup of level $\p^n$ of a Bianchi group $\sl(\O)$ where $\p$ is an
unramified prime ideal 
of $\O$. Fix $E$. Then 
\begin{enumerate}
\item if the residue degree of $\p$ is one, then 
$$ \dim H^1(\Gamma(\p^n),E) \ll p^{2n},$$
\item   if the residue degree of $\p$ is two, then 
$$ \dim H^1(\Gamma(\p^n),E) \ll p^{5n}$$
\end{enumerate}
as $n$ increases.
\end{theorem}
Note that the trivial upper bounds are $p^{3n}$ and $p^{6n}$ respectively. It is natural 
to look at these asymptotics from the perspective of the volume which is a topological 
invariant in our setting. Observe that the volume of $Y_{\Gamma(\p^n)}$ is given by 
a constant times the index of $\Gamma(\p^n)$ in the Bianchi group $\sl(\O)$. Thus 
asymptotically, the trivial asymptotic upper bound for the above cohomology groups 
is {\em linear} in the volume and the above upper bounds of Calegari and Emerton can be interpreted 
as {\em sublinear}.

Using the techniques discussed in this paper, we can derive the following lower bounds. 

\begin{proposition}\label{prop_asym}
Let $p$ be a rational prime that is unramified in $K$ and let $\Gamma(p^n)$ denote the  principal congruence subgroup of level $(p)^n$ of a Bianchi group $\sl(\O)$. 
\begin{enumerate} 
\item Then $$ \dim H^1_{cusp}(\Gamma(p^n),E_{k,k})\gg k$$
as $k$ increases and $n$ is fixed,
\item Assume further that $K$ is of class number one. Then 
$$\dim H^1_{cusp} (\Gamma(p^n),\C)\gg p^{3n}$$ 
as $n$ increases.
\end{enumerate}
\end{proposition}

\begin{proof} 
Recall from Proposition \ref{lefschetzlowerbound} that $${\rm dim} H^1_{cusp}(\Gamma, E_{k,k}) \geq \dfrac{1}{2}\bigg( L(\sigma,\Gamma,k) +\tr(\sigma^1_{Eis},\Gamma, k) - \tr(\sigma^2_{Eis},\Gamma, k)\bigg).$$ When $\Gamma$ is fixed, by Corollary \ref{cor_eisensteinbound} the dimension of the Eisenstein part of the cohomology is the same for every weight $k>0$. Hence, the asymptotic for (1) is given by Corollary \ref{cor_lefprin}.

The claim in (2) follows directly from Theorems \ref{trace1_eisenstein} and \ref{trace2_eisenstein}, together with the Lefschetz number formula provided in Corollary \ref{cor_lefprin}..
\end{proof}

%-----------------------------------------------------------------------------------------------------

\subsection{Lower bounds for $\gl$}

 In this section we will discuss the trace of  $\sigma$ on the cohomoogy of $\gl(\O)$. For convenience let us put $\Gamma=\sl(\O)$ and $G=\gl(\O)$. 

Let us start with a couple of observations. As $G=\Gamma \rtimes \langle \beta \rangle$ with $\beta:=( \begin{smallmatrix} -1 & 0 \\ 0&1 \end{smallmatrix})$ and $\beta$ acts trivially on the cusps of $\Gamma$, the groups $\Gamma$ and $G$ have the same cusps. Given a cusp $c$, its stabilizer in $G$  (modulo $\pm Id$) 
is of the form $\Z^2 \rtimes \Z/ 2\Z$. This implies that the connected components of the boundary of Borel-Serre compactification of $Y_G$ are 2-orbifolds whose underlying manifolds are 2-spheres. In turn, the cohomology of the boundary vanishes and we get 
$$H^1(G,E_{k,k})=H^1_{cusp}(G,E_{k,k}).$$ 

From the inflation-restriction sequence we see that 
$$H^1(G,E_{k,k}) = H^1(\Gamma,E_{k,k})^{\left \langle \beta \right \rangle}.$$ 
The involutions $\sigma^1$ and $\tau^1$ commute and $\sigma^1\tau^1$ equals the action of $\beta$. Hence we get 
$$H^1(G,E_{k,k}) = H^1(\Gamma,E_{k,k})^{\sigma^1\tau^1}.$$ 
Counting the dimensions of the common eigenspaces, we see by comparison that 
$$\tr(\tau^1, \Gamma,E_{k,k}) + \tr(\sigma^1, \Gamma,E_{k,k})= 2 \cdot \tr (\sigma^1,G, E_{k,k}).$$ 
The matrix $\beta$ acts on $E_{k,k}$ trivially and acts as $-Id$ on $H^1(\partial X_\Gamma,E_{k,k})$. This implies that 
$$\tr(\tau^1_{Eis}, \Gamma, E_{k,k}) = - \tr(\sigma^1_{Eis}, \Gamma, E_{k,k}).$$ 
Using this last identity, together with the previous facts, we get (dropping $E_{k,k}$ from the notation for convenience) 
$$L(\tau,\Gamma)+L(\sigma,\Gamma)=-4 \cdot \tr(\sigma^1,G)+\tr(\tau^0,\Gamma)+\tr(\sigma^0,\Gamma)+\tr(\tau_{Eis}^2,\Gamma)+\tr(\sigma_{Eis}^2,\Gamma).$$

Using results from previous sections, we get the following simplified formula for the trace of  $\sigma$ on $H^1(\gl(\O), E_{k,k})$.

\begin{theorem} \label{thm_tracegl}
Let $L(\tau, \sl(\O), E_{k,k})$ and $L(\sigma, \sl(\O), E_{k,k})$ be as in Theorem \ref{thm_leftau}. Then, 
$$\tr(\sigma^1\mid H^1(\gl(\O),E_{k,k}))=\dfrac{-1}{4} \bigg (L(\tau,\sl(\O),E_{k,k})+L(\sigma,\sl(\O),E_{k,k})+2^t-4\cdot \delta(k,0) \bigg )$$ 
where $t$ is the number of rational primes which ramify over $K$ and $\delta(k,0)$ is the Kronecker $\delta$-function as defined in Theorem  \ref{trace2_eisenstein}.
\end{theorem}

 Using Theorem \ref{thm_leftau} and the fact that 
$$\dim H^1(\gl(\O),E_{k,k}) \geq |\tr(\sigma^1,\gl(\O),E_{k,k})|,$$
we get the following asymptotics.
 
 \begin{corollary}\label{cor_asymgl}
 Let $D$ be the discriminant of $K/\Q$ and $\O_K$ be its ring of integers. As $K/\Q$ is fixed and $k\to \infty$, we have 
$$\dim H^1(\gl (\O_K),E_{k,k})\gg k$$ 
where the implicit constant depends on the discriminant $D$. As $k$ is fixed and $|D| \to \infty$, we have 
$$\dim H^1(\gl(\O_K),E_{k,k}) \gg \varphi (D)$$ 
where $\varphi$ is the Euler phi-function and the implicit constant depends on the weight $k$. 
 \end{corollary}

Note that one can write a more precise formula for the lower bounds above. As the formulas for the Lefschetz numbers are 
complicated, we stated our results in a slightly weaker form for the sake of simplicity .

%----------------------------------------------------------------------------------------------------------

%------------------------------------------------------------------------------------------------------------

\end{document}